\documentclass[10pt]{elsarticle}
\usepackage{amsfonts,amssymb,amsmath,epsfig}

\newtheorem{theorem}{Theorem}
\newtheorem{lemma}[theorem]{Lemma}
\newdefinition{definition}[theorem]{Definition}
\newdefinition{remark}[theorem]{Remark}
\newdefinition{example}[theorem]{Example}
\newproof{proof}{Proof}

\newcommand{\complex}{\mathbb{C}}

\makeatletter
\def\ps@pprintTitle{%
  \let\@oddhead\@empty
  \let\@evenhead\@empty
  \let\@oddfoot\@empty
  \let\@evenfoot\@oddfoot
}
\makeatother
\begin{document}



\bibliographystyle{plain}

\def\im{\mathop{\rm Im}\nolimits}
\def\diag{\mathop{\rm diag}\nolimits}
\def\re{\mathop{\rm Re}\nolimits}
\def\inv{\mathop{\rm Inv}\nolimits}
\def\ker{\mathop{\rm Ker}\nolimits}
\def\gker{\mathop{\rm GKer}\nolimits}
\def\sp{\mathop{\rm span}\nolimits}
\def\rank{\mathop{\rm rank}\nolimits}
\def\nrank{\mathop{\rm nrank}\nolimits}
\newcommand{\bi}[1]{\hbox{\boldmath{$#1$}}}
\newcommand{\CC}{\mathbb C}

\begin{frontmatter}
\title{On the singular two-parameter eigenvalue problem II\tnoteref{t1}}
\tnotetext[t1]{This work was supported by the Research Agency of the Republic of Slovenia, Research Grant P1-0154.}
\author[1]{Toma\v{z} Ko\v{s}ir}
\ead{tomaz.kosir@fmf.uni-lj.si}
\author[1]{Bor Plestenjak\corref{cor1}}
\ead{bor.plestenjak@fmf.uni-lj.si}
\cortext[cor1]{Corresponding author}
\address[1]{Faculty of Mathematics and Physics, University of Ljubljana,
Jadranska 19, SI-1000 Ljubljana, Slovenia}

\begin{abstract}
In the 1960s, Atkinson introduced an abstract algebraic setting for multiparameter
eigenvalue problems. He showed that a nonsingular multiparameter eigenvalue problem is equivalent
to the associated system of generalized eigenvalue problems, which is a key relation for
many theoretical results and numerical
methods for nonsingular multiparameter eigenvalue problems. In  2009, Muhi\v{c} and Plestenjak extended the above relation to a class of singular two-parameter eigenvalue problems with coprime characteristic polynomials and such that all finite eigenvalues are algebraically simple. They
introduced a way to solve a singular two-parameter
eigenvalue problem by computing the common regular eigenvalues of the associated system of two
singular generalized eigenvalue problems.
Using new tools, in particular the stratification theory, we extend this connection to singular two-parameter eigenvalue problems with possibly multiple eigenvalues and such that characteristic polynomials can have a nontrivial common factor.
\end{abstract}


\end{frontmatter}

\section{Introduction}

We consider \emph{the algebraic two-parameter eigenvalue problem}
\begin{equation}\label{problem}
\begin{split}
  W_1(\lambda,\mu)x_1&:=(A_1+\lambda B_1+\mu C_1)x_1=0,\\
  W_2(\lambda,\mu)x_2&:=(A_2+\lambda B_2+\mu C_2)x_2=0,
\end{split}\tag{W}
\end{equation}
where $A_i,B_i$, and $C_i$ are $n_i\times n_i$ matrices over $\complex$,
$\lambda,\mu\in\complex$, and $x_i\in \complex^{n_i}$.
If $(\lambda,\mu)$ and nonzero $x_1,x_2$
satisfy
(\ref{problem}), then the
tensor product $x_1\otimes x_2$ is the corresponding
\emph{(right) eigenvector}. Similarly, $y_1\otimes y_2$ is the
corresponding \emph{left eigenvector} if $y_i\ne 0$ and
$y_i^*W_i(\lambda,\mu)=0$ for $i=1,2$.
In the generic case, problem (\ref{problem}) has $n_1n_2$ \emph{eigenvalues} counting multiplicities $(\lambda,\mu)$ that are roots of the
system of two bivariate characteristic polynomials
\begin{equation}\label{detproblem}
\begin{split}
  p_1(\lambda,\mu)&:=\det(W_1(\lambda,\mu))=0,\\
  p_2(\lambda,\mu)&:=\det(W_2(\lambda,\mu))=0.
\end{split}
\end{equation}

A two-parameter eigenvalue problem (\ref{problem}) is related
to a coupled pair of
generalized eigenvalue problems
\begin{equation}\label{drugi}
\begin{split}
  \Delta_1 z&=\lambda \Delta_0 z,\\
  \Delta_2 z&=\mu \Delta_0 z,
\end{split}\tag{$\Delta$}
\end{equation}
where $n_1n_2\times n_1n_2$ matrices
\begin{equation}\label{Deltaik}
\begin{split}
  \Delta_0&=B_1\otimes C_2-C_1\otimes B_2,\\
  \Delta_1&=C_1\otimes A_2-A_1\otimes C_2,\\
  \Delta_2&=A_1\otimes B_2-B_1\otimes A_2
\end{split}
\end{equation}
are called \emph{operator determinants}
and $z=x_1\otimes x_2$ is a decomposable vector,
for details see, e.g., \cite{Atkinson}.
A generic two-parameter eigenvalue problem
(\ref{problem}) is \emph{nonsingular}, which means that the corresponding
operator determinant $\Delta_0$ is nonsingular. In this case (see,
e.g., \cite{Atkinson}), the matrices $\Delta_0^{-1}\Delta_1$ and $\Delta_0^{-1}\Delta_2$ commute and
the eigenvalues of (\ref{problem}) agree with the eigenvalues of (\ref{drugi}).
This relation enables us to solve a nonsingular two-parameter eigenvalue
problem numerically using standard tools for generalized
eigenvalue problems, see, e.g., \cite{HKP}.

Several applications lead to singular two-parameter eigenvalue
problems, where $\Delta_0$ is singular
and
(\ref{drugi}) is a pair of singular
generalized eigenvalue problems, see, e.g., the computation of the signed distance between two overlapping ellipsoids \cite{Iwata}, solving a system of two bivariate polynomials \cite{BiRoots}, an application in stochastic games \cite{Attia}, aerolastic flutter problem \cite{Pons}, and applications in \cite{MP2}.
 We assume that (\ref{problem}) is singular in the homogeneous setting, i.e., all linear combinations $\alpha_0 \Delta_0 +\alpha_1 \Delta_1+\alpha_2 \Delta_2$ are singular, as
otherwise the problems can be transformed into a nonsingular one by a linear substitution of parameters $\lambda$ and $\mu$.
In the singular case, the relation between problems (\ref{problem}) and (\ref{drugi}) is much less understood.
For a start, we need to define what is an eigenvalue
of (\ref{problem}) and (\ref{drugi}) in the singular case.

We assume that bivariate pencils $W_1(\lambda,\mu)$ and $W_2(\lambda,\mu)$
are both regular, i.e., $p_i(\lambda,\mu)\not\equiv 0$ for $i=1,2$.
The problem (\ref{problem}) is then singular if either
\begin{enumerate}
\item [a)] degree of $p_i$ is less than $n_i$ for $i=1,2$, or
\item [b)] polynomials $p_1$ and $p_2$ have a nontrivial common factor.
\end{enumerate}

\begin{definition}\label{def:eigW} {\rm
Let $(\lambda_0,\mu_0)\in\complex^2$ be a point that does not belong to
a nontrivial common factor of characteristic
polynomials $p_1$ and $p_2$. We say that $(\lambda_0,\mu_0)$ is a \emph{finite eigenvalue}
of (\ref{problem}) if there exist nonzero vectors $x_1$ and $x_2$ that
satisfy (\ref{problem}).
}
\end{definition}

It is easy to see that all isolated solutions of
the system of characteristic polynomials (\ref{detproblem}) are
finite eigenvalues of (\ref{problem}).

A definition of eigenvalue that covers both regular and singular pencils is that
 $\lambda_0\in\CC$ is an \emph{eigenvalue} of a matrix pencil $A-\lambda B$
if $\rank(A-\lambda_0 B)<\nrank(A-\lambda B)$, where
  $\nrank(A-\lambda B) := \max_{\xi \in \mathbb{C}} \rank(A -\xi B)$
  is the \emph{normal rank} of a matrix pencil.
In a similar way, $\infty$ is an eigenvalue if
  $\rank (B) < \nrank(A-\lambda B)$.

We define a common finite eigenvalue of a pair of matrix
pencils (\ref{drugi}) as follows.

\begin{definition}\label{def:eigDelta} {\rm
  A point $(\lambda_0, \mu_0)\in\complex^2$ is a \emph{(finite) eigenvalue}
  of  matrix pencils $\Delta_1 -\lambda \Delta_0$ and $\Delta_2-\mu \Delta_0$
  if the following is true:
  \begin{enumerate}
    \item[a)] $\rank(\Delta_1-\lambda_0 \Delta_0)<
      \nrank(\Delta_1-\lambda \Delta_0)$,
    \item[b)] $\rank(\Delta_2-\mu_0 \Delta_0)<
      \nrank(\Delta_2-\mu \Delta_0)$,
    \item[c)] there exists a common regular eigenvector
      $z$, i.e., a nonzero $z$ such that
      $(\Delta_1 -\lambda_0 \Delta_0) z  =  0$,
      $(\Delta_2-\mu_0 \Delta_0) z  =  0$,
      and $z\not\in {\cal R}(\Delta_i,\Delta_0)$ for $i=1,2$,
  where ${\cal R}(\Delta_i,\Delta_0)$ is the minimal reducing subspace (for a definition, see Section \ref{sec:kron}) for $\Delta_i$ and $\Delta_0$.
  \end{enumerate}
  }
\end{definition}

Under the assumption that $p_1$ and $p_2$ are coprime, i.e., they do not have a nontrivial
common factor, it is shown in \cite{MP2} that if all eigenvalues of (\ref{problem}) are
algebraically simple, then eigenvalues of (\ref{problem}) and (\ref{drugi}) agree.
Also, if $(\lambda_0, \mu_0)\in\complex^2$ is an algebraically simple eigenvalue of (\ref{problem}), then $(\lambda_0, \mu_0)$ is an eigenvalue of (\ref{drugi}).
Several numerical examples in \cite{MP2} indicate that the relation between the eigenvalues of (\ref{problem}) and (\ref{drugi}) could be extended to multiple eigenvalues
and also to the case where $p_1$ and $p_2$ have a nontrivial
common factor. In this paper we confirm these findings by providing the supporting theory.

In Section \ref{sec:kron} we introduce the Kronecker canonical form and other
tools. We introduce stratification together with some new related results in Section \ref{sec::strat}. In Section \ref{sec::eig} we show that if
$p_1$ and $p_2$ are coprime, then each
eigenvalue of (\ref{problem}) is an eigenvalue of (\ref{drugi}) and
vice versa. For problems
such that $p_1$ and $p_2$ do have a nontrivial common factor
 we extend the relation to
points that do not belong to a common factor. In
Section \ref{sec::common} we discuss conditions for a point
$(\lambda_0,\mu_0)$ on a common curve that would make it an eigenvalue and give
several illuminating numerical examples.

\section{Auxiliary results}\label{sec:kron}
 The Kronecker canonical form (KCF) of a matrix pencil is one of
 the main tools in this paper. For more details on the KCF and its numerical
computation see, e.g., \cite{guptri1}, \cite{guptri2},
\cite{Gantmacher}, and \cite{Van_Dooren_singular}.

\begin{definition} { \rm
  Let $A- \lambda B \in \mathbb{C}^{m \times n}$ be a matrix pencil.
  Then there exist nonsingular matrices $P \in \mathbb{C}^{m \times m}$ and
  $Q \in \mathbb{C}^{n\times n}$ such that
  \begin{displaymath}
    P^{-1} (A-\lambda B)Q = \widetilde{A}-\lambda \widetilde{B}
    = \diag(A_1 - \lambda B_1, \ldots, A_k - \lambda B_k)
  \end{displaymath}
  is in \emph{Kronecker canonical form}. Each block $A_i-\lambda B_i$,
  $i=1,\ldots,k$, has one of the following forms:
  $J_d(\alpha), N_d, L_d$, or $L_d^T$, where the matrices
  \begin{displaymath}
    J_d (\alpha) = \left[\begin{matrix}
      \alpha -\lambda & 1 & & \cr & \ddots & \ddots & \cr & & \ddots & 1\cr
      & & & \alpha -  \lambda\end{matrix}\right]\in\mathbb{C}^{d\times d}, \quad
    N_d  = \left[\begin{matrix}1 & -\lambda  & & \cr & \ddots & \ddots & \cr
      & & \ddots & -\lambda\cr & & & 1\end{matrix}\right]\in\mathbb{C}^{d\times d},
  \end{displaymath}
  \begin{displaymath}
    L_d = \left[\begin{matrix} -\lambda & 1 &  &\cr
      & \ddots & \ddots & \cr & & -\lambda & 1\end{matrix}\right]\in\mathbb{C}^{d\times (d+1)}, \quad
    L_d^T = \left[\begin{matrix} -\lambda & &\cr 1 &\ddots & \cr
      & \ddots & -\lambda\cr & & 1\end{matrix}\right]\in\mathbb{C}^{(d+1)\times d},
  \end{displaymath}
  represent a finite regular block, an infinite regular block, a right
  singular block, and a left singular block, respectively.
To each Kronecker block
we associate a \emph{Kronecker chain} of linearly independent vectors as
follows:

\begin{enumerate}
  \item[a)] A finite regular block $J_d(\alpha)$ is associated with
    vectors $u_1,\ldots,u_d$ that satisfy
    \begin{equation*}
    \begin{split}
       (A-\alpha B)u_1&=0,\\
       (A-\alpha B)u_{i+1}&=Bu_i,\quad i=1,\ldots,d-1.
    \end{split}
    \end{equation*}
  \item[b)] An infinite regular block $N_d$ is associated with
    vectors $u_1,\ldots,u_d$ that satisfy
    \begin{equation*}
    \begin{split}
       Bu_1&=0,\\
       Bu_{i+1}&=Au_i,\quad i=1,\ldots,d-1.
    \end{split}
    \end{equation*}
  \item[c)] A right singular block $L_d$ is associated with
     vectors $u_1,\ldots,u_{d+1}$ that satisfy
    \begin{equation*}
    \begin{split}
       Bu_1&=0,\\
       Bu_{i+1}&=Au_i,\quad i=1,\ldots,d,\\
       0&=Au_{d+1}.
    \end{split}
    \end{equation*}
  \item[d)] For $d\ge 1$, a left singular block $L_d^T$ is associated with
    vectors $u_1,\ldots,u_{d}$ that satisfy
    \begin{equation*}
    \begin{split}
      Bu_i&=Au_{i+1},\quad i=1,\ldots,d-1.
    \end{split}
    \end{equation*}
\end{enumerate} }
\end{definition}

The union of Kronecker chains for all Kronecker blocks
is a basis for $\complex^n$.

One can see that
$L_d \left[\begin{matrix}1 & \lambda & \cdots & \lambda^d\end{matrix}\right]^T=0$
and to each right singular block there corresponds one of linearly independent vectors from the kernel of $A-\lambda B$. If vectors $u_1,\ldots,u_{d+1}$ form a  Kronecker chain for a right singular block $L_d$, then
$$(A-\lambda B)(\lambda^d u_1 + \lambda^{d-1} u_2 +\cdots +\lambda u_d + u_{d+1})=0$$ for all $\lambda$. Thus, to each right singular block $L_d$ there corresponds a polynomial vector $p(\lambda)$ of degree $d$ such that
$(A-\lambda B)p(\lambda)=0$ for all $\lambda$.

We know that $\nrank(A-\lambda B)=n-s$, where $s\ge 0$ is the number of the right singular blocks
$L_d$ in the KCF of $A-\lambda B$. If $A-\lambda B$ is a singular
pencil then $s\ge 1$ and there exist $s$ linearly independent polynomial vectors
$p_1(\lambda),\ldots,p_s(\lambda)$, each corresponding to a different right singular block, such that for a generic $\lambda_0\in\CC$ vectors
$p_1(\lambda_0),\ldots,p_s(\lambda_0)$ form a basis for $\ker(A-\lambda_0 B)$.
We say that $\sp(p_1(\lambda),\ldots,p_s(\lambda))$ is the \emph{generic kernel} of $A-\lambda B$ and denote it by $\gker(A-\lambda B)$. For each $\lambda\in\CC$, $\dim(\gker(A-\lambda B))=s$ and $\gker(A-\lambda B)\subset \ker(A-\lambda B)$.

A subspace
${\cal M}\subset\complex^n$ is a \emph{reducing subspace} for
$A-\lambda B$ if $\dim(A{\cal M}+B{\cal M})=\dim({\cal M})-s$. The vectors from
the Kronecker chains of all right singular blocks $L_d$
form a basis for the \emph{minimal reducing subspace}
${\cal R}(A,B)$, which is a subset of all reducing subspaces.
${\cal R}(A,B)$ is unique and can be computed 
from the generalized upper triangular form (GUPTRI), see, e.g.,
\cite{guptri1,guptri2,MCS}. One can also compute ${\cal R}(A,B)$ numerically by
exploiting the relation
$${\cal R}(A,B)=\bigcup_{\xi \in\CC}\gker(A-\xi B).$$

If $\lambda_0\in\CC$ is an eigenvalue of a singular matrix pencil $A-\lambda B$, then the eigenvector cannot be defined in a usual way, but the following is true.

\begin{lemma} For a (singular) matrix pencil $A- \lambda B$ the following is equivalent:
\begin{enumerate}
\item[a)] $\lambda_0\in\CC$ is an eigenvalue of $A-\lambda B$,
\item[b)] $\rank(A-\lambda_0 B)<\nrank(A-\lambda B)$,
\item[c)] $\dim(\gker(A-\lambda_0 B))<\dim(\ker(A-\lambda_0 B))$,
\item[d)] there exists a nonzero $v$ such that $v\not\in\gker(A-\lambda_0 B)$ and
$(A-\lambda_0 B)v=0$,
\item[e)] there exists a nonzero $v$ such that $v\not\in {\cal R}(A,B)$ and
$(A-\lambda_0 B)v=0$.
\end{enumerate}
\end{lemma}

If $\lambda_0$ is an eigenvalue of $A-\lambda B$, then we say that each
nonzero vector $v$ such that $v\not\in{\cal R}(A,B)$ and
$(A-\lambda_0 B)v=0$ is a corresponding \emph{(right) regular eigenvector}.
Clearly, if we pick a random vector $v$ from $\ker(A-\lambda_0 B)$, then $v$ is a regular eigenvector with probability one.

For the study of (\ref{drugi}) we will apply the following result
from \cite{KosirKron} that shows that
the kernel of
an operator determinant $\Delta=A\otimes D - B\otimes C$ can be
constructed from the Kronecker chains of matrix
pencils $A-\lambda B$ and $C-\mu D$.

\begin{theorem}[{\cite[Theorem 4]{KosirKron}}]
  \label{thm:Tomaz}
  A basis for the kernel of $\Delta=A\otimes D - B\otimes C$
  is the union of sets of linearly independent
  vectors associated with the
  pairs of Kronecker blocks of the following types:
  \begin{enumerate}
    \item[a)] $(J_{d_1}(\alpha_1),\; J_{d_2}(\alpha_2))$, where
       $\alpha_1=\alpha_2$,
    \item[b)] $(N_{d_1},\; N_{d_2})$,
    \item[c)] $(N_{d_1},\; L_{d_2})$,
    \item[d)] $(L_{d_1},\; N_{d_2})$,
    \item[e)] $(L_{d_1},\; J_{d_2}(\alpha))$,
    \item[f)] $(J_{d_1}(\alpha),\; L_{d_2})$,
    \item[g)] $(L_{d_1},\; L_{d_2})$,
    \item[h)] $(L_{d_1},\; L_{d_2}^T)$, where $d_1<d_2$,
    \item[i)] $(L_{d_1}^T,\; L_{d_2})$, where $d_1>d_2$,
  \end{enumerate}
  where the left block of each pair belongs to the
  pencil $A-\lambda B$ and the right block belongs to the
  pencil $C-\mu D$.

  For each pair of Kronecker blocks that satisfies
  a) or b) we can construct an associated
  set of linearly independent
  vectors $z_1,\ldots,z_d$ in the kernel of $\Delta$
  as follows.
    Let the vectors $u_1,\ldots,u_{d_1}$ form
  a Kronecker chain associated with the
  block $J_{d_1}(\alpha_1)$ (or $N_{d_1}$)
   of the pencil $A-\lambda B$ and
  let the vectors $v_1,\ldots,v_{d_2}$ form
  a Kronecker chain associated with the
  block $J_{d_2}(\alpha_2)$ (or $N_{d_2}$) of the pencil $C-\mu D$.
  Then $d=\min(d_1,d_2)$ and
  $$z_j=\sum_{i=1}^ju_i\otimes v_{j+1-i},\quad j=1,\ldots,d.$$
\end{theorem}

In the above theorem
we omit the constructions of vectors in the kernel
for all pairs of Kronecker blocks that are irrelevant to our case.
For a complete description, see \cite{KosirKron}.
A similar technique is used to
describe the kernels of generalized Sylvester operators
in \cite{Demmel_Edelman}.

\begin{definition}For a pair of matrix pencils $A-\lambda B$ and $C-\mu D$ we  denote by $T(A-\lambda B,C-\mu D)$ the size of the kernel of
$A\otimes D-B\otimes C$.
\end{definition}

\section{Stratification}\label{sec::strat}

If we perturb a regular matrix pencil $A-\lambda B$ into $\widetilde A-\lambda
\widetilde B$, where $\widetilde A$ and $\widetilde B$ are arbitrarily close
to $A$ and $B$, respectively, then the stratification theory \cite{Edelman2, MCS} explains which are possible structures
of the KCF of the perturbed pencil, based on the KCF of the initial pencil.

Let $\alpha_1,\ldots,\alpha_k$ be different eigenvalues (including a possible  eigenvalue $\infty$) of $A-\lambda B$. For an eigenvalue $\alpha_j$ let $d_{j1}\ge d_{j2}\ge \cdots \ge d_{jm_j}$ be sizes of the corresponding Jordan blocks (or infinite blocks for eigenvalue $\infty$) for $j=1,\ldots,k$.

The KCF of $\widetilde A-\lambda \widetilde B$ can be obtained from the KCF of the
initial pencil $A-\lambda B$ using a finite sequence of the following two   moves:
\begin{itemize}
\item \emph{MLW (minimum leftward move)}: For one eigenvalue $\alpha$ we get from blocks of sizes $d_{1}\ge d_{2}\ge \cdots \ge d_{m}$ blocks of sizes
    $d'_{1}\ge d'_{2}\ge \cdots \ge d'_{m}$, where
    $d'_p=d_p+1$ and $d'_q=d_q-1$ for $p<q$ and
    $d_p=d_{p+1}=\cdots =d_{q-1}\ge d_q$, for an eigenvalue $\alpha'$ close to $\alpha$ (it can remain unchanged). The number of blocks corresponding to $\alpha'$ equals the number of blocks for $\alpha$, except when $q=m$ and
    $d_q=1$ in which case the perturbed pencil has $m-1$ blocks.
\item \emph{HC (horizontal cut)}: From one eigenvalue $\alpha$ and its blocks of sizes $d_{1}\ge d_{2}\ge \cdots \ge d_{m}$ we get two eigenvalues $\beta\ne\gamma$, both close to $\alpha$ (one of them can remain $\alpha$), such that the sizes of blocks for $\beta$
    are $b_{1}\ge b_{2}\ge \cdots \ge b_{m}$, where $b_j=\min(d_j,C)$, and the
    sizes of blocks for $\gamma$ are $c_{1}\ge c_{2}\ge \cdots \ge c_{m}$, where
    $c_j=d_j-b_j$ for $j=1,\ldots,m$ and $1\le C<d_1$ is the size of the cut. If $c_j=0$ for some $j\in\{1,\ldots,m\}$, then there are less than $m$ blocks for the eigenvalue $\gamma$.
\end{itemize}
In the above two moves, whenever $\alpha=\infty$, it can perturb into finite eigenvalues, while a finite eigenvalue cannot perturb into an infinite one. In addition to the above two moves, each eigenvalue can perturb into a nearby one without any change in the structure of Jordan (or infinite) blocks.

We say that matrix pencils $A-\lambda B$ and $C-\lambda D$ are equivalent if there exist nonsingular matrices $P$ and $Q$ such that
$C-\lambda D=P(A-\lambda B)Q$. The \emph{orbit} of a matrix pencil is the manifold of all equivalent pencils, which means that all pencils in the orbit have equal KCF. A \emph{bundle} is the union of all orbits with the same canonical structure (number of blocks and their sizes) but with the eigenvalues
unspecified. We denote the bundle that corresponds to the canonical structure of
the pencil $A-\lambda B$ by ${\cal B}(A-\lambda B)$. We say that bundle ${\cal B'}$ \emph{covers} bundle ${\cal B}$ if we get
${\cal B'}$ by applying a MLW or HC move on ${\cal B}$. More precisely, we say that
${\cal B'}$ is a \emph{MLW cover} of ${\cal B}$ or a \emph{HC cover} of ${\cal B}$.

In the following section we will study the relation between the eigenvalues of (\ref{problem}) and (\ref{drugi}). One condition for $(\lambda_0,\mu_0)$ to be an eigenvalue of (\ref{drugi}) is that $\rank(\Delta_1-\lambda_0\Delta_0)<\rank(\Delta_1-\xi \Delta_0)$ for a generic $\xi\ne \lambda_0$. It follows from
$$\Delta_1-\lambda \Delta_0 = W_1(\lambda,0)\otimes C_2 - C_1\otimes W_2(\lambda,0)$$
that a perturbation from $\lambda=\lambda_0$ to $\lambda=\xi$ changes the structure
of the KCF in Theorem \ref{thm:Tomaz}. To understand how
 this affects $\rank(\Delta_1-\lambda \Delta_0)$ it is enough to study what happens when we apply a MLW or HC move to the corresponding bundles. The following two lemmas provide the appropriate theory.

\begin{lemma}\label{lem:MLW}
Let $A-\lambda B$ and $C-\mu D$ be a pair of regular matrix pencils and let
$\widetilde A-\lambda \widetilde B$ and $\widetilde C-\mu \widetilde D$ be
their arbitrarily close perturbations, such that ${\cal B}(\widetilde A-\lambda \widetilde B)$ is a MLW
cover of ${\cal B}(A-\lambda B)$ and
${\cal B}(\widetilde C-\mu \widetilde D)={\cal B}(C-\mu D)$.
More precisely, let the MLW move be applied to an eigenvalue $\alpha$ of  $A-\lambda B$ and let
$d_1\ge d_2\ge \cdots \ge d_m$ be sizes of the corresponding blocks in the
KCF. If the MLW move affects blocks with indices $p$ and $q$, where $p<q$ and
    $d_p=d_{p+1}=\cdots =d_{q-1}\ge d_q$, then
\begin{enumerate}
\item $\widetilde T\le T$, where $T={\rm dim}(\ker(A\otimes D-B\otimes C))$ and $\widetilde T={\rm dim}(\ker(\widetilde A\otimes \widetilde D-\widetilde B\otimes \widetilde C))$.
\item If the KCF of $C-\mu D$ contains a block $J_e(\alpha)$ such that $d_p\ge e\ge d_q$, then $\widetilde T<T$,
\end{enumerate}
\end{lemma}

\begin{proof}
A result of a generic perturbation would be that
$\widetilde A-\lambda \widetilde B$ and $\widetilde C-\mu \widetilde D$ do not have
any common eigenvalues and thus $\widetilde T=0$. In our analysis we therefore assume that the eigenvalues of perturbed pencils $\widetilde A-\lambda \widetilde B$ and $\widetilde C-\mu \widetilde D$ match in a way that gives the maximum possible $\widetilde T$.

Let $e_1\ge e_2\ge \cdots \ge e_\ell$
be the sizes of blocks corresponding to $\alpha$ in the
KCF of the matrix pencil $C-\mu D$. From Theorem \ref{thm:Tomaz} it follows that
these blocks contribute
\begin{equation}\label{T_alpha}
T(\alpha)=\sum_{i=1}^m\sum_{j=1}^\ell \min(d_i,e_j)
\end{equation}
linearly independent vectors to $\ker(A\otimes D-B\otimes C)$. If
$\alpha$ perturbs into $\alpha'$ in the MLW move, then we assume that $\alpha'$ is also an eigenvalue of $\widetilde C-\mu \widetilde D$ and that
the KCF of $\widetilde C-\mu \widetilde D$ contains blocks
$J_{e_q}(\alpha')$ for $q=1,\ldots,\ell$.

What changes in the MLW move are sizes of the blocks. After the MLW move the
blocks that correspond to $\alpha'$ in $\widetilde A-\lambda \widetilde B$ have sizes
    $d'_{1}\ge d'_{2}\ge \cdots \ge d'_{m}$, where
    $d'_p=d_p+1$, $d'_q=d_q-1$ and $d'_s=d_s$ for $s\ne p,q$. It follows that
$$T-\widetilde T=\sum_{j=1}^\ell (\min(d_p,e_j)-\min(d_p+1,e_j))
+\sum_{j=1}^\ell (\min(d_q,e_j)-\min(d_q-1,e_j)).$$
The difference in the first sum is 1 for indices where $d_p< e_j$ and
the difference in the second sum is -1 for indices where $e_j\ge d_q$.
Since $d_p\ge d_q$, we clearly have $\widetilde T\le T$, and  if there
exists $e_j$ such that $d_p\ge e_j\ge d_q$, then $\widetilde T<T$.
\end{proof}

\begin{lemma}\label{lem:HC}
Let $A-\lambda B$ and $C-\mu D$ be a pair of regular matrix pencils and let
$\widetilde A-\lambda \widetilde B$ and $\widetilde C-\mu \widetilde D$ be
their arbitrarily close perturbations such that ${\cal B}(\widetilde A-\lambda \widetilde B)$
is a HC cover of ${\cal B}(A-\lambda B)$ and
${\cal B}(\widetilde C-\mu \widetilde D)$ is a HC cover of ${\cal B}(C-\mu D)$.
Then $\widetilde T\le T$, where $T={\rm dim}(\ker(A\otimes D-B\otimes C))$ and $\widetilde T={\rm dim}(\ker(\widetilde A\otimes \widetilde D-\widetilde B\otimes \widetilde C))$.
\end{lemma}

\begin{proof}
A HC move splits an eigenvalue into two nearby eigenvalues. As only the eigenvalues
that pencils $A-\lambda B$ and $C-\mu D$ have in common (including the infinite ones) contribute to $T$, we only have to consider a situation
when we apply a HC move to a common eigenvalue $\alpha$.

If $d_1\ge \cdots \ge d_m$ and $e_1\ge \cdots \ge e_\ell$ are sizes of the blocks corresponding to $\alpha$ in the KCF of $A-\lambda B$ and
$C-\mu D$, respectively, then these blocks contribute $T(\alpha)$ from (\ref{T_alpha}) to $T$. If we perform the HC move only to one of the pencils, then clearly $\widetilde T\le T$. It is easy to see that we get maximum possible $\widetilde T$ when we apply the HC move to both pencils in such way that eigenvalue $\alpha$ simultaneously
splits into two common eigenvalues $\alpha_1$ and $\alpha_2$ of $\widetilde A-\lambda \widetilde B$ and $\widetilde C-\mu \widetilde D$.

It is enough to study the situation for just one pair of blocks. Suppose that the block $J_{d}(\alpha)$ in $A-\lambda B$ splits into blocks
$J_{d_1}(\alpha_1)$ and $J_{d_2}(\alpha_2)$, and similarly, the block $J_{e}(\alpha)$ in $C-\mu D$ splits into blocks
$J_{e_1}(\alpha_1)$ and $J_{e_2}(\alpha_2)$, such that $d=d_1+d_2$ and $e=e_1+e_2$. Without loss of generality we can assume that $d\le e$.
Then
$$\min(d_1,e_1)+\min(d_2,e_2)\le d_1+d_2 = d = \min(d,e).$$

As the above inequality holds for all possible pairs of one block from $A-\lambda B$ and the other one from $C-\mu D$, it follows that $\widetilde T\le T$.
\end{proof}
\section{Relation between the eigenvalues of (\ref{problem}) and (\ref{drugi})}\label{sec::eig}

In the nonsingular case the eigenvalues of (\ref{problem}) agree
with the eigenvalues of the associated pair of generalized
eigenvalue problems (\ref{drugi}), see, e.g., \cite{Atkinson}.
In this section we show that in the singular case in a similar way the finite
regular eigenvalues of (\ref{problem}) are related
to the finite regular eigenvalues of (\ref{drugi}). While this was established
for algebraically simple eigenvalues of (\ref{problem}) with coprime
 polynomials $p_1$ and $p_2$ in \cite{MP2}, we extend the relation
in this paper to all eigenvalues.

First we show that we can assume that the configuration of eigenvalues is as general as possible.

\begin{lemma}\label{lem:predpostavke}
It is always possible to apply a linear substitution of parameters
$\lambda$ and $\mu$ in (\ref{problem}) in such way that
if
$(\lambda_0,\mu_0)\in\CC^2$ is an eigenvalue of (\ref{problem}) as in Definition
\ref{def:eigW} then the following is true:
\begin{enumerate}
\item Matrix pencil $W_i(\lambda_0,0)-\gamma C_i$ is regular for $i=1,2$.
\item Matrix pencil $W_i(0,\mu_0)-\beta B_i$ is regular for $i=1,2$.
\item In a small neighbourhood of $\lambda_0$ it holds that $\rank(W_i(\eta,\mu_0))=n_i$ for all $\eta\ne \lambda_0$ and
$i=1,2$.
\item In a small neighbourhood of $\mu_0$ if holds that $\rank(W_i(\lambda_0,\theta))=n_i$ for all $\theta\ne \mu_0$ and $i=1,2$.
\item If $(\lambda_1,\mu_1)\ne (\lambda_0,\mu_0)$ is another eigenvalue of (\ref{problem})
then $\lambda_1\ne \lambda_0$ and $\mu_1\ne  \mu_0$.
\item In a small
neighbourhood of $(\lambda_0,\mu_0)$ it holds for all $(\eta,\theta)\ne (\lambda_0,\mu_0)$ that
$$\sum_{i=1}^2\rank(W_i(\lambda_0,\mu_0))< \sum_{i=1}^2\rank(W_i(\eta,\theta)).$$
\end{enumerate}
\end{lemma}

\begin{proof}
We can apply a substitution in the form of an orthogonal rotation
\begin{equation}\label{eq:rot_sub}\left[\begin{matrix}\lambda\cr \mu\end{matrix}\right]=
\left[\begin{matrix}\cos \phi & -\sin \phi\cr \sin \phi & \cos \phi\end{matrix}\right]
\left[\begin{matrix}\widetilde \lambda\cr \widetilde \mu\end{matrix}\right]=:
R(\phi)
\left[\begin{matrix}\widetilde \lambda\cr \widetilde \mu\end{matrix}\right]
,\end{equation}
where $\phi\in[0,2\pi)$.
\begin{enumerate}
\item We know that problem (\ref{problem}) has less than $n_1n_2$ eigenvalues. Therefore, for a generic $\phi$ the rotation
(\ref{eq:rot_sub})
leads to eigenvalues $(\widetilde\lambda,\widetilde\mu)$ such that different eigenvalues differ in both components.
\item Matrix pencil $W_i(\lambda_0,0)-\gamma C_i$ is singular if $p_i(\lambda_0,\gamma)\equiv 0$, which means that $p_i(\lambda,\mu)$ is divisible by $\lambda-\lambda_0$. As $p_i$ for $i=1,2$ can have only finitely many such divisors, none of them is present after a substitution
    (\ref{eq:rot_sub}) for a generic $\phi$.
\item Similar as above.
\item The assumption is that $W_i(\lambda,\mu)$ is regular, i.e., ${\rm nrank}(W_i(\lambda,\mu))=n_i$ for $i=1,2$. Eigenvalue $(\lambda_0,\mu)$ lies on one or more curves $q_i(\lambda,\mu)=0$, where $q_i(\lambda,\mu)$ divides $p_i(\lambda,\mu)$.
    If we use a generic rotation $\phi$ than none of these curves has a vertical tangent at $(\lambda_0,\mu_0)$ and
    point 4. holds.
\item Similar as above.
\item This property does not depend on the substitution, but follows from the assumption that $(\lambda_0,\mu_0)$ is an isolated solution of $p_1(\lambda,\mu)=p_2(\lambda,\mu)=0$. Therefore, there exists a neighbourhood of
    $(\lambda_0,\mu_0)$ that does not include other eigenvalues and does not intersect with any curve $q(\lambda,\mu)=0$, where
    $q(\lambda,\mu)$ is a common factor of $p_1(\lambda,\mu)$ and $p_2(\lambda,\mu)$.
\end{enumerate}
\end{proof}

\begin{theorem}\label{thm:WtoDelta} Let $(\lambda_0,\mu_0)\in\CC^2$ be an eigenvalue of (\ref{problem})
and let all assumptions from Lemma \ref{lem:predpostavke} be satisfied. Then
$(\lambda_0,\mu_0)$ is an eigenvalue of the associated problem (\ref{drugi}).
\end{theorem}

\begin{proof}
In order to obtain a basis for $\ker(\Delta_1-\eta \Delta_0)$ from Theorem \ref{thm:Tomaz},
we apply the equality
\begin{equation}\label{eq:zveza}
\Delta_1-\eta \Delta_0 = W_1(\eta,0)\otimes C_2 - C_1\otimes W_2(\eta,0).
\end{equation}
It follows from item 3 in Lemma \ref{lem:predpostavke} that $W_i(\eta,0)-\gamma C_i$ is regular for
$\eta$ sufficiently close to $\lambda_0$  and its
KCF contains only blocks of types $J_d$ and $N_d$ for $i=1,2$.
We will show that $\rank(\Delta_1-\lambda_0 \Delta_0)<
      \nrank(\Delta_1-\eta \Delta_0)$ for all $\eta\ne\lambda_0$ sufficiently close to $\lambda_0$.

We know from Theorem \ref{thm:Tomaz} that
a basis for $\ker(\Delta_1-\lambda_0 \Delta_0)$ is the union of sets of linearly
independent  vectors associated with the pairs of Kronecker blocks of the type
$(J_{d_1}(\gamma_1),\; J_{d_2}(\gamma_2))$, where
       $\gamma_1=\gamma_2$, or
    $(N_{d_1},\; N_{d_2})$.
If we change $\lambda_0$ to $\eta$ by an infinitesimally small perturbation, then $W_i(\eta,0)-\gamma C_i$ is still
regular for $i=1,2$. A basis for $\ker(\Delta_1-\eta \Delta_0)$ thus depends
on pairs of Kronecker blocks of the type
$(J_{d_1}'(\gamma_1),\; J_{d_2}'(\gamma_2))$, where
       $\gamma_1=\gamma_2$, or
    $(N_{d_1}',\; N_{d_2}')$, where $J_d'$ and $N_d'$ are blocks from
the KCF of $W_i(\eta,0)-\gamma C_i$ for $i=1,2$.

We now apply the stratification theory from Section \ref{sec::strat}. It follows
that one gets from the bundle of $W_i(\lambda_0,0)-\gamma C_i$ to the bundle  of
$W_i(\eta,0)-\gamma C_i$ by a finite sequence of the moves MLW and HC introduced in Section \ref{sec::strat}.
There exists a finite sequence of pairs of bundles
$({\cal B}_{1j},{\cal B}_{2j})$, $j=0,\ldots,k$, such that
${\cal B}_{i0}={\cal B}(W_i(\lambda_0,0)-\gamma C_i)$,
${\cal B}_{ik}={\cal B}(W_i(\eta,0)-\gamma C_i)$, and
either ${\cal B}_{i,j+1}={\cal B}_{ij}$
or ${\cal B}_{i,j+1}$ covers ${\cal B}_{ij}$ for $i=1,2$
and $j=0,\ldots,k-1$. In addition, the eigenvalues in ${\cal B}_{i,j+1}$ are  arbitrarily
close to the eigenvalues in ${\cal B}_{ij}$, which means that, due to the continuity of the eigenvalues,
${\cal B}_{i,j+1}$ cannot have less distinct eigenvalues as
${\cal B}_{ij}$.

Let us denote $T_0={\rm dim}(\ker(\Delta_1-\lambda_0 \Delta_0))$ and
$T_k={\rm dim}(\ker(\Delta_1-\eta \Delta_0))$. We know from Theorem \ref{thm:Tomaz} that we can compute $T_0$ and $T_k$ from the canonical structures of
$W_i(\lambda_0,0)-\gamma C_i$ and $W_i(\eta,0)-\gamma C_i$ for $i=1,2$. In a similar way
we define $T_j$ for $j=1,\ldots,k-1$ as the number of linearly independent vectors that we get for the pair of bundles $({\cal B}_{1j},{\cal B}_{2j})$ from matching blocks in ${\cal B}_{1j}$ and ${\cal B}_{2j}$ as in Theorem \ref{thm:Tomaz}, where we assume that the eigenvalues of ${\cal B}_{1j}$ match the eigenvalues
${\cal B}_{2j}$ in a way that gives the maximum possible $T_j$.

Lemmas \ref{lem:MLW} and \ref{lem:HC} show that $T_{j+1}\le T_j$ for $j=0,\ldots,k-1$. For a MLW move it is enough to consider the situation when we apply MLW only to one bundle in a pair as we can combine a sequence of two individual changes into one step. This is done in Lemma \ref{lem:MLW}. On the other hand, if we apply a HC move then we assume that a HC move was applied to both bundles in such way that we get the same pair of eigenvalues in both bundles, otherwise clearly
$T_{j+1}\le T_{j}$. This situation is covered in Lemma \ref{lem:HC}. If follows that $T_k\le T_0$.

We can also show that $T_k<T_0$. In the beginning, pencils $W_i(\lambda_0,0)-\gamma C_i$, $i=1,2$, have a common eigenvalue $-\mu_0$ and combinations
of the form $(J_{d_1}(-\mu_0),J_{d_2}(-\mu_0))$ contribute at least one linearly independent vector to
$\ker(\Delta_1-\lambda_0 \Delta_0)$. In the end pencils $W_i(\eta_0,0)-\gamma C_i$, $i=1,2$, do not have a common eigenvalue $\theta$
close to $-\mu$ as this would mean that $(\eta,-\theta)$ close to $(\lambda_0,\mu_0)$ and $\eta\ne \lambda_0$ is an eigenvalue of
(\ref{problem}) which is in contradiction with $(\lambda_0,\eta_0)$ being an isolated solution of (\ref{detproblem}).
Therefore, $T_k<T_0$ and
$\rank(\Delta_1-\lambda_0 \Delta_0)<
      \nrank(\Delta_1-\eta \Delta_0)$.
This means that $\lambda_0$ is an eigenvalue of $\Delta_1-\lambda \Delta_0$. In the same way we can show that
$\mu_0$ is an eigenvalue of $\Delta_2-\mu \Delta_0$.

What remains to show is that there exists a common regular eigenvector that
satisfies Definition \ref{def:eigDelta}.
If we take a nonzero $z\in\ker(W_1(\lambda_0,\mu_0))\otimes \ker(W_2(\lambda_0,\mu_0))$, then
it is clear from the above that $z\in \ker(\Delta_1-\lambda_0 \Delta_0)\cap \ker(\Delta_2-\mu_0 \Delta_0)$ but
$z\not\in{\rm GKer}(\Delta_1-\lambda_0 \Delta_0)$ and $z\not\in{\rm GKer}(\Delta_2-\mu_0 \Delta_0)$. This shows
that $z$ is a common regular eigenvector and $(\lambda_0,\mu_0)$ is an eigenvalue of (\ref{drugi}).
\end{proof}

Let us emphasize that Theorem \ref{thm:WtoDelta} also covers the case when $p_1$ and $p_2$ have a nontrivial common factor.
We showed that, regardless to the existence of a common factor, each eigenvalue of (\ref{problem}) as defined in Definition \ref{def:eigW} is an eigenvalue of
(\ref{drugi}). To show the implication in the other direction, we require that $p_1$ and $p_2$ are coprime. The following theorem
is a generalization of Theorem 3.7 in \cite{MP2} where it is in addition required that all eigenvalues of (\ref{problem}) are algebraically simple.

\begin{theorem}\label{thm:DeltatoW} Let $(\lambda_0,\mu_0)\in\CC^2$ be an eigenvalue of a pair of matrix pencils (\ref{drugi}) associated to
a (singular) two-parameter eigenvalue problem (\ref{problem}) that satisfies all assumptions in Lemma \ref{lem:predpostavke}.
If the characteristic polynomials of (\ref{problem}) are coprime, then $(\lambda_0,\mu_0)$ is an eigenvalue of (\ref{problem}).
\end{theorem}

\begin{proof}As the proof is a slight modification of the proof of Theorem 3.7 in \cite{MP2}, we will omit most of the details.
We start with a nonzero $z\in\ker(\Delta_1-\lambda_0\Delta_0)\cap\ker(\Delta_2-\mu_0\Delta_0)$, which is a regular eigenvector
for the eigenvalue $(\lambda_0,\mu_0)$ and
apply the relation
\begin{equation}\label{eq:zveza2}
\Delta_1-\lambda \Delta_0 = W_1(\lambda,0)\otimes C_2 - C_1\otimes W_2(\lambda,0).
\end{equation}
If follows from Theorem \ref{thm:Tomaz} that $z$ is a linear combination of vectors associated with
pairs $(J_{d_1}(\alpha),J_{d_2}(\alpha))$ or $(N_{d_1},N_{d_2})$ of pencils
$W_1(\lambda,0)-\gamma C_1$ and $W_2(\lambda,0)-\gamma C_2$. All other combinations from Theorem \ref{thm:Tomaz} are not possible due
to Lemma \ref{lem:predpostavke}.

In a similar way as in \cite{MP2} we can show that $z$ cannot be a linear combination of vectors solely from the combinations of
type $(N_{d_1},N_{d_2})$. Namely, if this would be true, then $z$ would belong to the minimal reducing subspace ${\cal R}(\Delta_2,\Delta_0)$ which is in contradiction with $z$ being a regular eigenvector for the eigenvalue $\mu_0$ of $\Delta_2-\mu \Delta_0$.

Therefore, $z$ must include a contribution from at least one Kronecker pair of type $(J_{d_1}(\gamma_0),J_{d_2}(\gamma_0))$, where $\gamma_0$ is a common eigenvalue of $W_1(\lambda,0)-\gamma C_1$ and $W_2(\lambda,0)-\gamma C_2$.
As the characteristic polynomials of (\ref{problem}) are coprime, it follows that $(\lambda_0,-\gamma_0)$ is an eigenvalue of (\ref{problem}) and then by Theorem \ref{thm:WtoDelta} $(\lambda_0,-\gamma_0)$ is also an eigenvalue of (\ref{drugi}). In a similar way we can show that there exists an eigenvalue $(-\beta_0,\mu_0)$ of (\ref{problem}) and then
$(-\beta_0,\mu_0)$ is an eigenvalue of (\ref{drugi}) as well. Since (\ref{problem}) has only finitely many eigenvalues, it follows that (\ref{drugi}) has finitely many eigenvalue as well.

A generic substitution (\ref{eq:rot_sub}) used in the proof of Lemma \ref{lem:predpostavke} affects the problem
(\ref{drugi}) and its eigenvalues as well. Thus we can assume that item 5 from Lemma \ref{lem:predpostavke} holds for the eigenvalues of (\ref{drugi}) as well. But, then the only option is that $\gamma_0=-\mu_0$ and $\beta_0=-\lambda_0$ which means that
$(\lambda_0,\mu_0)$ is an eigenvalue of (\ref{problem}).
\end{proof}

In numerical experiments we often noticed that pencils $\Delta_1-\lambda \Delta_0$ and $\Delta_2-\mu \Delta_0$ belong to the same bundle, i.e., they have the same canonical structure regarding blocks and their sizes. The following lemma explains the observation.

 \begin{lemma}\label{lem:same_bundle}
 Let us assume that we applied a generic orthogonal rotation (\ref{eq:rot_sub}) to a two-parameter eigenvalue problem (\ref{problem}). Then the associated matrix pencils $\Delta_1-\lambda \Delta_0$ and $\Delta_2-\mu \Delta_0$ belong to the same bundle.
 \end{lemma}
\begin{proof}We consider a family of pencils
\begin{equation}\label{eq:pencil_varphi}\cos\varphi \Delta_1 + \sin\varphi \Delta_2 - \lambda \Delta_0.
\end{equation}
As $\varphi$ goes from $0$ to $2\pi$, pencil (\ref{eq:pencil_varphi}) moves between bundles and returns back to the initial bundle at $\varphi=2\pi$. For all except finitely many values $\varphi$  pencil (\ref{eq:pencil_varphi}) belongs to the same bundle, which we can call the generic bundle
for (\ref{eq:pencil_varphi}). Since pencils $\Delta_1-\lambda \Delta_0$ and $\Delta_2-\mu \Delta_0$ are just special cases of
(\ref{eq:pencil_varphi}) for $\varphi=0$ and $\varphi=\pi/2$, respectively, they generically belong to the same bundle.
\end{proof}

\begin{example} We take example 4.4 from \cite{MP2}, where
$$W_1(\lambda,\mu)=\left[\begin{matrix}2+\lambda & 1+2 \lambda & \lambda \cr
 \lambda & 2+2\lambda+2\mu & \mu \cr  \mu & 1+2\mu & 2+\mu \end{matrix}\right]$$
 and
  $$W_2(\lambda,\mu)=\left[\begin{matrix}1+\lambda & 1+2\lambda & \lambda \cr
  \lambda & 1+2\lambda+2\mu & \mu \cr \mu & 1+2\mu & 1+\mu\end{matrix}\right].
$$
It turns out (for details see \cite{MP2}) that $(\lambda_1,\mu_1)=(-{1\over 2},-{1\over 2})$ is a quadruple eigenvalue
of both (\ref{problem}) and (\ref{drugi}), but due to the multiplicity an implication in either side could not be deducted from the theory available at the time. The extended theory in this paper
covers the example. As the characteristic polynomials of $W_1$ and $W_2$ are coprime, we can apply Theorems \ref{thm:WtoDelta} and \ref{thm:DeltatoW} to establish that the eigenvalues of (\ref{problem}) and (\ref{drugi}) do agree.
\end{example}

\begin{example} If we apply linearization 1 from \cite{BiRoots}
to bivariate polynomials
\begin{equation}\label{bivar_problem}
\begin{split}
  p_1(\lambda,\mu)&=(x^2+y^2-2)(x+y-1),\\
  p_2(\lambda,\mu)&=(x^2+y^2-2)(x-y+1),
\end{split}
\end{equation}
we get determinantal representations
$$W_1(\lambda,\mu)=\left[\begin{matrix}2-2\lambda-2\mu & -\lambda & \lambda + \mu & 0 & \lambda+\mu-1 \cr
\lambda & -1 & 0 & 0 & 0\cr
0 & \lambda & -1 & 0 & 0 \cr
\mu & 0 & 0 & -1 & 0\cr
0 & 0 & 0 & \mu & -1\end{matrix}\right]$$
and
$$W_2(\lambda,\mu)=\left[\begin{matrix}-2-2\lambda+2\mu & \lambda & \lambda - \mu & 0 & \lambda-\mu+1 \cr
\lambda & -1 & 0 & 0 & 0\cr
0 & \lambda & -1 & 0 & 0 \cr
\mu & 0 & 0 & -1 & 0\cr
0 & 0 & 0 & \mu & -1\end{matrix}\right]$$
such that $\det(W_i(\lambda,\mu))=p_i(\lambda,\mu)$ for $i=1,2$. The obtained two-parameter eigenvalue problem is singular by construction and, in addition, $p_1$ and $p_2$ have a common nontrivial factor. Theorem \ref{thm:WtoDelta} still holds for the eigenvalue $(\lambda_1,\mu_1)=(0,1)$ of the corresponding problem (\ref{problem}), which is an isolated solution of $p_1(\lambda,\mu)=p_2(\lambda,\mu)=0$. Point $(\lambda_1,\mu_1)$ is the only eigenvalue of the
corresponding problem (\ref{drugi}) and in this case eigenvalues of (\ref{problem}) and (\ref{drugi}) agree as well.
\end{example}

In the next section we explore the situation where (\ref{drugi}) has an eigenvalue on the common nontrivial factor of $\det(W_1(\lambda,\mu))$ and $\det(W_2(\lambda,\mu))$.

\section{Points on a nontrivial common curve}\label{sec::common}

Suppose that we have a singular two-parameter eigenvalue problem where pencils $W_1(\lambda,\mu)$ and $W_2(\lambda,\mu)$ are
both regular and characteristic polynomials $p_1$ and $p_2$ have a common nontrivial factor. For such problems we know
from Theorem \ref{thm:WtoDelta} that if $(\lambda_0,\mu_0)$ is an eigenvalue of (\ref{problem}) then $(\lambda_0,\mu_0)$ is also an eigenvalue of (\ref{drugi}). In the other direction, an eigenvalue of (\ref{drugi}) is just a candidate for an eigenvalue
of (\ref{problem}). We will show that a point $(\lambda_0,\mu_0)$ that belongs to a common factor of $p_1$ and $p_2$ can be an eigenvalue of (\ref{drugi}), but it is not clear if it is possible to extend Definition \ref{def:eigW} to include points on the
common factors as well.

\begin{example}\label{ex:common1}We take a two-parameter eigenvalue problem
$$W_1(\lambda,\mu)=\left[\begin{matrix}\lambda+\mu & \cr & \lambda-\mu\end{matrix}\right],\quad
  W_2(\lambda,\mu)=\left[\begin{matrix}\lambda+\mu-2 & \cr & \lambda-\mu\end{matrix}\right].
$$
The problem is clearly singular since $\lambda-\mu$ is a common factor of characteristic polynomials
$p_1$ and $p_2$. Matrices of the corresponding problem (\ref{drugi}) are
$$\Delta_0=\left[\begin{matrix}0 & & & \cr & -2 & &\cr & & 2 & \cr & & & 0\end{matrix}\right],\quad
\Delta_1=\left[\begin{matrix}-2 & & & \cr & 0 & &\cr & & 2 & \cr & & & 0\end{matrix}\right],\quad
\Delta_2=\left[\begin{matrix}2 & & & \cr & 0 & &\cr & & 2 & \cr & & & 0\end{matrix}\right].
$$
The canonical forms of $\Delta_1-\lambda \Delta_0$ and $\Delta_2-\mu \Delta_0$
are both equal to  $L_0\oplus L_0^T\oplus J_1(1)\oplus J_1(0)\oplus N_1$ and thus
have the same structure as predicted by Lemma \ref{lem:same_bundle}. The minimal reducing subspace, which is the same for
both pencils, is ${\rm Lin}(e_4)$.
Problem (\ref{drugi}) has two eigenvalues $(0,0)$ and $(1,1)$ with regular eigenvectors $e_2$ and $e_3$, respectively.

Both points $(0,0)$ and $(1,1)$ lie on the zero set of the common factor curve $\lambda-\mu=0$. If we look at a situation at the point $(0,0)$, we can see that the KCF of $W_1(0,0)-\gamma C_1$ and $W_2(0,0)-\gamma C_2$ is
$2J_1(0)$ and $J_1(-2)\oplus J_1(0)$ respectively. This gives $\dim(\ker(\Delta_1))=2$ by Theorem \ref{thm:Tomaz}. If we perturb pencils into $W_1(\varepsilon,0)-\gamma C_1$ and $W_2(\varepsilon,0)-\gamma C_2$, then the KCF changes to
$J_1(\varepsilon)\oplus J_1(-\varepsilon)$ and $J_1(-\varepsilon)\oplus J_1(-2+\varepsilon)$, respectively, which gives
$\dim(\ker(\Delta_1-\varepsilon\Delta_0))=1$. This means that $\lambda_0=0$ is an eigenvalue
of $\Delta_1-\lambda \Delta_0$. In a similar way we could show that $\mu_0=0$ is an eigenvalue
of $\Delta_2-\mu \Delta_0$ and then finally check that $e_2$ is a regular eigenvector for $(0,0)$.

Using the same approach we could show that $(1,1)$ is an eigenvalue of (\ref{drugi}). Both points $(0,0)$ and $(1,1)$ are intersections of a common curve and another curve from $p_1$ or $p_2$. This might look like a good condition that makes a point an eigenvalue of (\ref{drugi}), but the next example immediately  gives a counterexample for this assumption.
\end{example}

\begin{example}\label{ex:common2}We slightly change the first pencil from Example \ref{ex:common1} and use
$$W_1(\lambda,\mu)=\left[\begin{matrix}\lambda+\mu & 1 \cr & \lambda-\mu\end{matrix}\right],\quad
  W_2(\lambda,\mu)=\left[\begin{matrix}\lambda+\mu-2 & \cr & \lambda-\mu\end{matrix}\right].
$$
Point $(0,0)$ is still an intersection of the common curve $\lambda=\mu$ and curve $\lambda=-\mu$ from $W_1$, but now
 $(0,0)$ is not an eigenvalue of the corresponding problem (\ref{drugi}), whose matrices are
$$\Delta_0=\left[\begin{matrix}0 & & & \cr & -2 & &\cr & & 2 & \cr & & & 0\end{matrix}\right],\quad
\Delta_1=\left[\begin{matrix}-2 & & -1& \cr & 0 & & 1\cr & & 2 & \cr & & & 0\end{matrix}\right],\quad
\Delta_2=\left[\begin{matrix}2 & & 1 & \cr & 0 & & 1\cr & & 2 & \cr & & & 0\end{matrix}\right].
$$
The KCF of $\Delta_i-\lambda \Delta_0$ is $L_0\oplus L_1^T\oplus J_1(1)\oplus N_1$ for $i=1,2$.
If we look at a situation at $(0,0)$, we see that the KCF of $W_1(0,0)-\gamma C_1$ is
$J_2(0)$ and the KCF of $W_1(\varepsilon,0)-\gamma C_1$ is
$J_1(\varepsilon)\oplus J_1(-\varepsilon)$, while canonical forms related to $W_2$ are same as in Example \ref{ex:common1}.
We get $\dim(\ker(\Delta_1))=\dim(\ker(\Delta_1-\varepsilon\Delta_0))=1$, which confirms that
 $0$ is not an eigenvalue of $\Delta_1-\lambda \Delta_0$.
\end{example}

In the above two examples we were investigating an intersection point between the common factor and a curve belonging
to one of the pencils. The examples show that this alone is not enough for a point to be an eigenvalue of  (\ref{drugi}). In the following examples situation is even more complicated as $p_1=p_2$.

\begin{example}\label{ex:common3}We take
$$W_1(\lambda,\mu)=\left[\begin{matrix}\lambda+\mu & \lambda \cr & \lambda+\mu\end{matrix}\right],\quad
  W_2(\lambda,\mu)=\left[\begin{matrix}\lambda+\mu & \cr & \lambda+\mu\end{matrix}\right].
$$
Now $p_1(\lambda,\mu)=p_2(\lambda,\mu)=(\lambda+\mu)^2$. It turns out that
$(0,0)$ is a double eigenvalue of (\ref{drugi}), whose matrices are
$$\Delta_0=\left[\begin{matrix}0& & 1& \cr & 0 &  & 1\cr & & 0 & \cr & & & 0\end{matrix}\right],\quad
\Delta_1=0,\quad
\Delta_2=0.
$$
The KCF of $\Delta_i-\lambda \Delta_0$ is $2L_0\oplus 2L_0^T\oplus 2J_1(0)$ for both $i=1,2$. The minimal reducing subspace, which is the same for
both pencils, has dimension 2 and is spanned by $e_1$ and $e_2$, while
$e_3$ and $e_4$ are regular eigenvectors for the eigenvalue $(0,0)$.

If we look at a situation at $(0,0)$, we see that the KCF of $W_i(0,0)-\gamma C_i$ is
$2 J_1(0)$ for $i=1,2$. The perturbed pencils $W_i(\varepsilon,0)-\gamma C_i$ for $i=1,2$ have canonical forms
$J_2(\varepsilon)$ and $2J_1(\varepsilon)$, respectively. This shows that
$4=\dim(\ker(\Delta_1))>\dim(\ker(\Delta_1-\varepsilon \Delta_0))=2$ and confirms that $0$ is a double eigenvalue
of $\Delta_1-\lambda \Delta_0$. In a similar way we show that $0$ is a double eigenvalue of
$\Delta_2-\mu \Delta_0$. As $e_3$ and $e_4$ are regular eigenvectors for both pencils, $(0,0)$ is a double eigenvalue
of (\ref{drugi}).

If we look back to the pencils $W_1$ and $W_2$ from (\ref{problem}) and try to figure out, what makes point $(0,0)$ an eigenvalue,
we can observe that $0=\rank(W_1(0,0))<\nrank(W_1)=2$, $0=\rank(W_2(0,0))<\nrank(W_2)=2$, and
$\rank(W_1(0,0))+\rank(W_2(0,0))< \rank(W_1(\theta,\eta))+\rank(W_2(\theta,\eta))$ for all $(\theta,\eta)\ne (0,0)$. All other points $(\lambda_0,-\lambda_0)$, where $\lambda_0\ne 0$, on the common factor satisfy
$\rank(W_1(\lambda_0,-\lambda_0))<\nrank(W_1)$ and $\rank(W_2(\lambda_0,-\lambda_0))<\nrank(W_2)$, but fail to satisfy the condition
$\rank(W_1(\lambda_0,-\lambda_0))+\rank(W_2(\lambda_0,-\lambda_0))< \rank(W_1(\theta,\eta))+\rank(W_2(\theta,\eta))$ for all $(\theta,\eta)\ne (\lambda_0,\eta_0)$ close to $(\lambda_0,-\lambda_0)$. But, as we will see in the next example, this alone is not enough for a point to be an eigenvalue.
\end{example}

\begin{example}\label{ex:common4}We slightly change the second pencil from Example \ref{ex:common3} into
$$W_1(\lambda,\mu)=\left[\begin{matrix}\lambda+\mu & \lambda \cr & \lambda+\mu\end{matrix}\right],\quad
  W_2(\lambda,\mu)=\left[\begin{matrix}\lambda+\mu & 1 \cr & \lambda+\mu\end{matrix}\right].
$$
The result of the change is that (\ref{drugi}) has no eigenvalues.
The corresponding $\Delta$ matrices are
$$\Delta_0=\left[\begin{matrix}0& & 1& \cr & 0 &  & 1\cr & & 0 & \cr & & & 0\end{matrix}\right],\quad
\Delta_1=\left[\begin{matrix}0& 1& & \cr & 0 &  & \cr & & 0 & 1\cr & & & 0\end{matrix}\right],\quad
\Delta_2=\left[\begin{matrix}0& -1& & -1\cr & 0 &  & \cr & & 0 & -1 \cr & & & 0\end{matrix}\right].
$$
The KCF of $\Delta_i-\lambda \Delta_0$ is
$L_0\oplus L_1\oplus L_0^T\oplus L_1^T$ for $i=1,2$.

We can observe that $\rank(W_1(0,0))<\nrank(W_1)$, $\rank(W_2(0,0))<\nrank(W_2)$, and
$\rank(W_1(0,0))+\rank(W_2(0,0))< \rank(W_1(\theta,\eta))+\rank(W_2(\theta,\eta))$ for all $(\theta,\eta)\ne (0,0)$,
which makes a point $(0,0)$ special, but not special enough to be an eigenvalue.
\end{example}

\begin{example}\label{ex:common5}In the last example $0$ is an eigenvalue of $\Delta_i-\lambda \Delta_0$ for $i=1,2$, but there does not exist a common regular eigenvector that would make $(0,0)$ an eigenvalue of (\ref{drugi}). We take
\begin{equation*}
\begin{split}
W_1(\lambda,\mu)&=\left[\begin{matrix}\lambda + \mu & 1 & \lambda & \cr & \lambda+\mu & & \cr
& & \lambda+\mu & 1 \cr & & & \lambda+\mu\end{matrix}\right],\\
W_2(\lambda,\mu)&=\left[\begin{matrix}\lambda + \mu & 1 & & \cr & \lambda+\mu & & \cr
& & \lambda+\mu & 1 \cr & & & \lambda+\mu\end{matrix}\right].
\end{split}
\end{equation*}
The KCF of
$\Delta_i-\lambda \Delta_0$ (we leave the computation of the corresponding $\Delta$ matrices of size $16\times 16$ to the reader)
is $6L_0^T\oplus 2L_1\oplus 4L_0\oplus 2J_1(0)\oplus 6N_1$ for $i=1,2$.

To obtain the generic kernel and the minimal reducing subspace for the pencil $\Delta_1-\lambda \Delta_0$ we study a pair of matrix pencils
$W_i(\lambda,0)-\gamma C_i$ such that $\lambda\ne 0$ for $i=1,2$. The KCF of $W_1(\lambda,0)-\gamma C_1$ has a block
$J_3(\lambda)$ with a Kronecker chain $\lambda e_1,e_3,e_4$ and a block $J_1(\lambda)$ with eigenvector $-\lambda e_2+e_3$.
The KCF of $W_2(\lambda,0)-\gamma C_2$ has two blocks $J_2(\lambda)$ with chains $e_1,e_2$ and $e_3,e_4$. It follows from
Theorem \ref{thm:Tomaz} that
 the generic kernel of $\Delta_1-\lambda \Delta_0$ has dimension 6 and
${\rm GKer}(\Delta_1-\lambda \Delta_0)={\rm Lin}(e_1\otimes e_1,e_1\otimes e_3,e_1\otimes e_2+e_2\otimes e_1,
e_1\otimes e_4+e_2\otimes e_3, \lambda e_1\otimes e_2 + e_3\otimes e_1,
\lambda e_1\otimes  e_4+e_3\otimes e_3)$.

In a similar way we get that
 the generic kernel of $\Delta_2-\mu \Delta_0$ has dimension 6 as well and
${\rm GKer}(\Delta_2-\mu \Delta_0)={\rm Lin}(e_1\otimes e_1,e_1\otimes e_3,e_1\otimes e_2+e_2\otimes e_1,
e_1\otimes e_4+e_2\otimes e_3, -\mu e_1\otimes e_2 + e_3\otimes e_1,
-\mu e_1\otimes  e_4+e_3\otimes e_3)$.
This shows that both pencils have the same minimal reducing subspace.
${\cal R}(\Delta_1,\Delta_0)={\cal R}(\Delta_2,\Delta_0)={\rm Lin}(e_1\otimes e_1,e_1\otimes e_2, e_1\otimes e_3,e_1\otimes e_4, e_2\otimes e_1,e_2\otimes e_3,e_3\otimes e_1, e_3\otimes e_3).$

In point $\lambda=0$ and $\mu=0$ both pencils have the same generic kernel
${\cal G}:={\rm Lin}(e_1\otimes e_1,e_1\otimes e_3,e_1\otimes e_2+e_2\otimes e_1,
e_1\otimes e_4+e_2\otimes e_3, e_3\otimes e_1,
e_3\otimes e_3)$. But, although nullspaces $\ker(\Delta_1)$ and $\ker(\Delta_2)$ both have dimension $8$,
they intersect in the common generic kernel and do
not have a common regular vector.
 A detailed analysis shows that
$\ker(\Delta_1)={\cal G}\oplus{\rm Lin}(e_3\otimes e_2 + e_4\otimes e_1,e_3\otimes e_4 + e_4\otimes e_3)$ and
$\ker(\Delta_2)={\cal G}\oplus{\rm Lin}(-e_1+e_3)\otimes e_2 + e_4\otimes e_1,(-e_1+e_3)\otimes e_4 + e_4\otimes e_3)$,
therefore $\ker(\Delta_1)\cap \ker(\Delta_2)={\cal G}$ and $(0,0)$ is not an eigenvalue of (\ref{drugi}).
\end{example}

\section{Conclusion}\label{sec::conc}
We extend the equivalence of eigenvalues of a two-parameter eigenvalue problem
(\ref{problem}) and the eigenvalues of the associated pair of generalized eigenvalue problem
(\ref{drugi}) to all singular problems where pencils $W_1(\lambda,\mu)$ and $W_2(\lambda,\mu)$ are regular
and $\det(W_1(\lambda,\mu))$ and $\det(W_2(\lambda,\mu))$ do not have a nontrivial common factor. This was first established for nonsingular problems by Atkinson in \cite{Atkinson} and then generalized in \cite{MP2} for singular problems
such that all eigenvalues are algebraically simple.

For the case where $\det(W_1(\lambda,\mu))$ and $\det(W_2(\lambda,\mu))$ do have a nontrivial common factor,
we show the implication in one direction, which makes the eigenvalues of (\ref{drugi}) candidates for the
eigenvalues of (\ref{problem}). We give several examples in Section \ref{sec::common} and, while the eigenvalues of (\ref{drugi}) are well defined regardless to the presence of common factors, it is still open how to extend the definition of eigenvalues of (\ref{problem}) to cover the points from the common factors so that they would agree to the eigenvalues of (\ref{drugi}).

Based on numerical experiments we believe that is should be possible to generalize
Theorems \ref{thm:WtoDelta} and \ref{thm:DeltatoW} to singular multiparameter eigenvalue problems with more than two parameters as well. However,
the proof would have to be based on a different approach as Theorem \ref{thm:Tomaz}, which is our key tool, cannot be generalized to three or more parameters in an elegant way.

\section*{Acknowledgements}
Both authors have been supported by the Research Agency of the Republic of Slovenia, Research Grant P1-0154.


\end{document}